\documentclass[letterpaper, 10 pt, conference]{ieeeconf}  

\IEEEoverridecommandlockouts                              

\usepackage{amsmath,amssymb,bm,bbm,mathrsfs,amscd}
\usepackage{color}
\usepackage{dsfont}
\usepackage{graphicx}
\usepackage{epsfig}
\usepackage{algorithm}
\usepackage[noend]{algpseudocode}
\usepackage{tikz}
\usepackage{hhline}
\usepackage{booktabs}
\usepackage{pifont}

\newtheorem{theorem}{Theorem}

\newtheorem{lemma}{Lemma}
\newtheorem{corollary}{Corollary}
\newtheorem{remark}{Remark}
\newtheorem{assumption}{Assumption}

\newtheorem{standassumption}{Standing Assumption}

\newcommand{\RR}{{\mathbb{R}}}
\newcommand{\NN}{{\mathbb{N}}}
\newcommand{\EE}{{\mathbb{E}}}
\newcommand{\PP}{{\mathbb{P}}}
\newcommand{\JJ}{{\mathbb{J}}}
\newcommand{\FF}{{\mathbb{F}}}

\newcommand{\mc}{\mathcal}
\newcommand{\norm}[1]{\left\|#1\right\|}
\newcommand{\normsq}[1]{\left\|#1\right\|^2}

\newcommand{\EEk}[1]{\EE\left[#1|\mc F^k\right]}

\newcommand{\op}{\operatorname}

\newcommand{\prox}{\operatorname{prox}}

\newcommand{\bs}{\boldsymbol}
\newcommand{\fineass}{\hfill\small$\blacksquare$}

\newcommand{\cmark}{\ding{51}}%
\newcommand{\xmark}{\ding{55}}%

\begin{document}

\title{Forward--Backward algorithms for stochastic Nash equilibrium seeking in restricted strongly and strictly monotone games}

\author{Barbara Franci and Sergio Grammatico
\thanks{The authors are in the Delft Center for System and Control, TU Delft, The Netherlands. E-mail addresses:
        {\tt\small \{b.franci-1, s.grammatico\}@tudelft.nl}.}
\thanks{This work was partially supported by NWO under research projects OMEGA (613.001.702) and P2P-TALES (647.003.003), and by the ERC under research project COSMOS (802348).}}

\markboth{Journal of \LaTeX\ Class Files,~Vol.~14, No.~8, August~2015}%
{Shell \MakeLowercase{\textit{et al.}}: Bare Demo of IEEEtran.cls for IEEE Journals}

\maketitle
\thispagestyle{empty}
\pagestyle{empty}

\begin{abstract}
We study stochastic Nash equilibrium problems with expected valued cost functions whose pseudogradient satisfies restricted monotonicity properties which hold only with respect to the solution. We propose a forward-backward algorithm and prove its convergence under restricted strong monotonicity, restricted strict monotonicity and restricted cocoercivity of the pseudogradient mapping. To approximate the expected value, we use either a finite number of samples and a vanishing step size or an increasing number of samples with a constant step. 
Numerical simulations show that our proposed algorithm might be faster than the available algorithms.
\end{abstract}

\IEEEpeerreviewmaketitle

\section{Introduction}

In 1950, John Nash came up with the idea of what we now call \textit{Nash equilibrium} \cite{nash1950}, a situation where none of the agents involved can improve its performance, given the actions of the other agents. Since then, the interest on the topic has only been rising \cite{facchinei2007,sandholm2010,pavel2019,yi2019,koshal2013,ravat2011}. In a Nash equilibrium problem (NEP), a number of agents interact to minimize their cost function, taking into consideration also the action of the other involved parties. A natural extension of this problem is to include some uncertainty, leading to stochastic NEPs (SNEPs). As an example, one can consider the energy market where the companies do not know the demand in advance \cite{henrion2007} or an application related to machine learning. In fact, in Generative Adversarial Nets \cite{goodfellow2014}, two neural network compete against each other in a two-players SNEP in order to, e.g., generate realistic images \cite{franci2020tnnls,gidel2018}. 

In these problems, the uncertainty formally translates in a random variable, i.e., in an expected value cost function. When the distribution of the random variable is known, the problem can be solved as in the deterministic case. In fact, the stochastic Nash equilibria (SNE) can be obtained as the solution of a suitable stochastic variational inequality (SVI) depending on the pseudogradient mapping of the game \cite{facchinei2007vi,facchinei2007,facchinei2010}. On the other hand, computing the exact expected value, even when possible, can be hard or computationally expensive. For this reason, we resort to the so-called stochastic approximation (SA) scheme, i.e., a way to estimate the expected value via a finite number of realizations of the random variable \cite{koshal2013,yousefian2017,iusem2017}. The challenging aspect is therefore that a stochastic error is committed at each iteration of a seeking algorithm. To control such error, one can either consider to take a huge number of samples \cite{iusem2017,franci2020fbtac,bot2020} or a vanishing step size sequence \cite{koshal2013,yousefian2017}. 

Due to the connection with SVIs, a number of algorithms are present in the literature that show convergence to a SNE. 
Among others, one can consider the extragradient (EG) algorithm \cite{kannan2019}, which involves two evaluations of the pseudogradient mapping. Instead, to obtain computationally lighter methods, one can consider adding a relaxation or regularization step. In this direction, the available algorithms include the Tikhonov (TIK) regularization \cite{koshal2013}, the regularized smoothed stochastic approximation (RSSA) scheme \cite{yousefian2017} and the stochastic projected reflected gradient (SPRG) method \cite{cui2016}.
The first three algorithms have the advantage of converging when the involved mappings are monotone while the latter requires also the weak sharpness property (a consequence of cocoercivity). In contrast, they have the disadvantage of being slow (due to the regularization) or computationally heavy (due to the double pseudogradient evaluation at each iteration). Perhaps, the fastest and simplest algorithm known is the stochastic forward-backward (SFB) algorithm \cite{bau2011,franci2020fbtac,yi2019}, which, in its first SA formulation, dates back to 1951 \cite{robbins1951} and which originated all the algorithms just introduced as variants and refinements. We refer to Table \ref{table_algo} for the convergence rates and an overall comparison between these methods and their properties. 

\begin{table}[h]
\begin{center}
\begin{tabular}{p{.9cm}ccccc}
\toprule
        & TIK \cite{koshal2013} & RSSA \cite{yousefian2017} & SPRG \cite{cui2016} & SEG \cite{kannan2019}& SFB  \\ 
\midrule
\textsc{Mon.}   &   \cmark  &  \cmark & \xmark   &    \cmark  &   \xmark \\\midrule

\textsc{Res.}   & \xmark   & \xmark  & \xmark & \xmark   & \cmark\\ \midrule

\textsc{NoReg.}   & \xmark   & \xmark  & \xmark & \cmark   & \cmark\\ \midrule

\textsc{\# Prox}& 1   & 1  & 1  & 2    & 1\\\midrule

\textsc{Rate}& -   & $\mc O\left(\tfrac{1}{\sqrt{K}}\right)$  & $\mc O\left(\tfrac{1}{\sqrt{K}}\right)$  & $\mc O\left(\frac{1}{K}\right)$    & $\mc O\left(\frac{1}{K}\right)$\\
\bottomrule
\end{tabular}
\end{center}
\caption{Known algorithms that converge under only monotonicity (Mon.) are marked with \cmark, which also indicates if the property is restricted to the solution only (Res.) and if the method does not involve a regularization or relaxation step (Reg.). \# Prox indicates the number of proximal or projection steps at each iteration. The last line refers to the convergence rates. }\label{table_algo}
\end{table}

The classic SFB is known to converge for strongly monotone \cite{franci2020fbecc,rosasco2016} or cocoercive mappings \cite{franci2020fbtac}. The main result of this paper is to show that such properties only with respect to the solution are sufficient to obtain convergence. 
Specifically, our contributions are the following
\begin{itemize}
\item The SFB algorithm converges to a SNE when the pseudogradient mapping is restricted strongly monotone, using the SA scheme with only one sample of the random variable, which is a computationally light approximation.
\item The convergence result holds also if the pseudogradient is restricted strictly monotone.
\item We propose an algorithm that is distributed and, since it allows for a vanishing step size sequence, no coordination among the agents is necessary.
\item We show convergence in the case of nonsmooth cost functions, thus extending some known results to a more general case.
\item Instead of one sample of the random variable, taking a higher and time-varying number of realizations can be considered and convergence is still guaranteed. In this case, restricted cocoercivity can be used.
\end{itemize}

We remark that considering restricted properties is important for instance when a limited amount of information is available to the agents. In fact, in partial decision information generalized (S)NEPs, the properties of the operators involved hold only at the solution, even when the pseudogradient mapping is strongly monotone \cite{gadjov2019,franci2020partial}. In this case, having a restricted property applies also to mere monotonicity \cite{gadjov2019,bianchi2020}. Thus, all other algorithms in Table \ref{table_algo} cannot be directly applied.

\subsection{Notation and preliminaries}

We use Standing Assumptions to postulate technical conditions that implicitly hold throughout the paper while Assumptions are postulated only when explicitly called.

$\RR$ denotes the set of real numbers and $\bar\RR=\RR\cup\{+\infty\}$.
$\langle\cdot,\cdot\rangle:\RR^n\times\RR^n\to\RR$ denotes the standard inner product and $\norm{\cdot}$ represents the associated Euclidean norm. 
Given $N$ vectors $x_{1}, \ldots, x_{N} \in \RR^{n}$, $\boldsymbol{x} :=\op{col}\left(x_{1}, \dots, x_{N}\right)=\left[x_{1}^{\top}, \dots, x_{N}^{\top}\right]^{\top}.$

The set of fixed points of $F$ is $\op{fix}(F):=\{x\in\RR^n\mid x\in F(x)\}$. 
For a closed set $C \subseteq \RR^{n},$ the mapping $\op{proj}_{C} : \RR^{n} \to C$ denotes the projection onto $C$, i.e., $\op{proj}_{C}(x)=\op{argmin}_{y \in C}\|y-x\|$. The residual mapping is, in general, defined as $\op{res}(x^k)=\norm{x^k-\op{proj}_{C}(x^k-F(x^k))}.$ Given a proper, lower semi-continuous, convex function $g$, the subdifferential is the operator $\partial g(x):=\{u\in\Omega \mid (\forall y\in\Omega):\langle y-x,u\rangle+g(x)\leq g(y)\}$. The proximal operator is defined as $\prox_{g}(v):=\op{argmin}_{u\in\Omega}\{g(u)+\frac{1}{2}\norm{u-v}^{2}_{}\}=\mathrm{J}_{\partial g}(v)$.
$\iota_C$ is the indicator function of the set C, i.e., $\iota_C(x)=1$ if $x\in C$ and $\iota_C(x)=0$ otherwise. The set-valued mapping $\mathrm{N}_{C} : \RR^{n} \to \RR^{n}$ denotes the normal cone operator for the the set $C$ , i.e., $\mathrm{N}_{C}(x)=\varnothing$ if $x \notin C,\left\{v \in \RR^{n} | \sup _{z \in C} v^{\top}(z-x) \leq 0\right\}$ otherwise.

We now recall some basic properties of operators \cite{facchinei2007}. 
A mapping $F:\op{dom}F\subseteq\RR^n\to\RR^n$ is: $\mu$-strongly monotone with $\mu>0$ if for all $x, y \in \op{dom}(F)$ $\langle F(x)-F(y),x-y\rangle\geq\mu\normsq{x-y}$; (strictly) monotone if for all $x, y \in \op{dom}(F)$ $(x\neq y)$ $\langle F(x)-F(y),x-y\rangle\geq\,(>)\,0;$ $\beta$-cocoercive with $\beta>0$, if for all $x, y \in \op{dom}(F)$ $\langle F(x)-F(y),x-y\rangle \geq \beta\|F(x)-F(y)\|^{2};$ firmly non expansive if for all $x, y \in \op{dom}(F)$ $\|F(x)-F(y)\|^{2} \leq\|x-y\|^{2}-\|(\mathrm{Id}-F) (x)-(\mathrm{Id}-F) (y)\|^{2};$ $\ell$-Lipschitz continuous if, for some $\ell>0$ $\norm{F(x)-F(y)} \leq \ell\norm{x-y} \text { for all } x, y \in \RR^{n}.$
We use the adjective \textit{restricted} if a property holds for all $(x,y)\in\op{dom}(F)\times \op{fix}(F)$.
We note that a firmly non expansive operator is also cocoercive, hence monotone and non expansive \cite[Definition 4.1]{bau2011}.

%


\section{Stochastic Nash equilibrium problem}\label{sec_SNEPs}

In this section we describe the stochastic Nash equilibrium problem (SNEP).
We consider a set $\mc I=\{1,\dots,N\}$ of noncooperative agents who interact with the aim of minimizing their cost function. Each of them has a decision variable $x_i\in\Omega_i\subseteq\RR^{n_i}$ where $\Omega_i$ indicates the local feasible set of agent $i\in\mc I$. Let us set $n=\sum_{i\in\mc I}n_i$ and $\bs x=\op{col}(x_1,\dots,x_N)\in\bs\Omega=\Omega_1\times,\dots,\times\Omega_N\subseteq\RR^n$. The cost function of agent $i$ is defined as 
\begin{equation}\label{eq_cost}
\JJ_i(x_i,\bs{x}_{-i}):=\EE_{\xi_i}[f_i(x_i,\bs{x}_{-i},\xi_i(\omega))] + g_{i}(x_{i}),
\end{equation}
where $\bs x_{-i}=\op{col}(\{x_j\}_{j\neq i})$ collects the decision variable of all the other agents. Let $(\Xi, \mc F, \PP)$ be the probability space and $\Xi=\Xi_1\times \dots\Xi_N$. Then, $\xi_i:\Xi_i\to\RR^d$ is a random variable and $\EE_{\xi_i}$ is the mathematical expectation with respect to the distribution of $\xi_i$\footnote{From now on, we use $\xi$ instead of $\xi(\omega)$ and $\EE$ instead of $\EE_\xi$.} for all $i\in\mc I$. 

\begin{standassumption}\label{ass_cost}
For each $i \in \mc I$ and $\boldsymbol{x}_{-i} \in \Omega_{-i}$ the function $f_{i}(\cdot, \boldsymbol{x}_{-i})$ is convex and continuously differentiable.
%
For each $i\in\mc I$ and for each $\xi \in \Xi$, the function $f_{i}(\cdot,\boldsymbol x_{-i},\xi)$ is convex, $\ell_i$-Lipschitz continuous, and continuously differentiable. The function $f_{i}(x_i,\bs x_{-i},\cdot)$ is measurable and for each $\boldsymbol x_{-i}\in \Omega_{-i}$ and the Lipschitz constant $\ell_i(\boldsymbol x_{-i},\xi)$ is integrable in $\xi$.
\fineass\end{standassumption}

The cost function in \eqref{eq_cost} is given by the sum of a smooth part, i.e., the expected value of the measurable function $f_i:\RR^n\times\RR^d\to\RR$, and a non-smooth part, given by the function $g_i:\RR^{n_i}\to\RR$. The latter may represent some local constraints via an indicator function, $g_i=\iota_{\Omega_i}$ or some penalty function to enforce a desired behavior.

\begin{standassumption}\label{ass_G}
For each $i\in\mc I$, the function $g_i$ in \eqref{eq_cost} is lower semicontinuous and convex and $\op{dom}(g_{i})=\Omega_i$ is nonempty, compact and convex.
\fineass\end{standassumption}

Each agent $i\in\mc I$ aims at solving its local optimization problem 
\begin{equation}\label{eq_game}
\forall i \in \mc I: \quad \min\limits _{x_i \in \Omega_i}  \JJ_i\left(x_i, \bs{x}_{-i}\right),
\end{equation}
given the decision variable of the other agents $\bs x_{-i}$.
Specifically, the goal is to reach a stochastic Nash equilibrium (SNE), namely, a situation where none of the agents can further decrease  its cost function given the decision variables of the others. Formally, a SNE is a collective decision $\bs x^*$ such that for all $i\in\mc I$
$$\JJ_i(x_i^{*}, \boldsymbol x_{-i}^{*}) \leq \inf \{\JJ_i(y, \boldsymbol x_{-i}^{*})\; | \; y \in \Omega_i\}.$$

Since the SNE correspond to the solutions of a suitable variational problem \cite[Proposition 1.4.2]{facchinei2007}, let us define the stochastic variational inequality (SVI) associated to the game in \eqref{eq_game}. First, let us introduce the pseudogradient mapping
\begin{equation}\label{eq_grad}
\FF(\bs{x})=\op{col}\left((\EE[\nabla_{x_{i}} f_{i}(x_{i}, \bs{x}_{-i},\xi)])_{i\in\mc I}\right),
\end{equation} 
where we exchange the expected value and the gradient symbol as a consequence of Standing Assumption \ref{ass_cost} \cite[Lemma 3.4]{ravat2011}.
Then, the associated SVI reads as
\begin{equation}\label{eq_svi}
\langle \FF(\bs x^*),\bs x-\bs x^*\rangle+\sum_{i\in\mc I}\{g_i(x_i)-g_i(x^*_i)\}\geq 0,\forall  \bs x \in \bs{\Omega}.
\end{equation}
and the stochastic variational equilibrium (v-SNE) of game in (\ref{eq_game}) is defined as the solution of the $\op{SVI}(\bs\Omega , \FF,G)$ in (\ref{eq_svi}) where $\FF$ is described in (\ref{eq_grad}) and $G(\bs x)=\partial g_1(x_1)\times\dots\times\partial g_N(x_N)$. 

\section{Distributed stochastic forward-backward algorithm with one sample}
In this section, we report the assumption for convergence of our proposed algorithm and present our first two results.
The iterations of the distributed forward-backward (FB) algorithm are presented in Algorithm \ref{algo}. 

\begin{remark}
When the local cost function is the indicator function, we can use the projection on the local feasible set $\Omega_i$, instead of the proximal operator \cite[Example 12.25]{bau2011}.
\fineass\end{remark}

\begin{algorithm}[t]
\caption{Distributed Stochastic Forward--Backward}\label{algo}
Initialization: $x_i^0 \in \Omega_i$\\
Iteration $k$: Agent $i$ receives $x_j^k$ for all $j \in \mathcal{N}_{i}^{h}$, then updates:
$$x_i^k=\op{prox}_{g_i}[x_i^k-\gamma^k_{i}\hat F_{i}(x_i^k, \boldsymbol{x}_{-i}^k,\xi_i^k)]$$
\end{algorithm}

First of all, since the expected value pseudogradient may be hard or impossible to compute in closed form, we estimate $\FF$ in \eqref{eq_grad} using a stochastic approximation (SA) scheme, i.e., we take only one realization of the random variable $\xi$: 
\begin{equation}\label{eq_F_SA}
\begin{aligned}
\hat F(\bs x^k,\bs \xi^k)&=F_\textup{SA}(\bs x^k,\bs \xi^k)\\
&=\op{col}\left(\nabla_{x_1}f_1(\bs x^k,\xi^k_1),\dots,\nabla_{x_N}f_N(\bs x^k,\xi^k_N)\right),
\end{aligned}
\end{equation}
where $\bs\xi^{k} =\op{col}(\xi_1^{k},\dots,\xi_{N}^k)$ is a collection of i.i.d. random variables.
Since we take an approximation, let the stochastic error be indicated with
$$\epsilon^k=F_\textup{SA}(\bs x^k,\bs\xi^k)-\FF(\bs x^k).$$
Then, we suppose that the approximation is unbiased.

\begin{assumption}[Zero mean error]\label{ass_error}
For al $k\geq 0$, a.s. $\EEk{\epsilon^k}=0$.
\fineass\end{assumption}

Besides the zero mean, one should also consider some conditions on the variance of the error. Specifically, we postulate an assumption on the step size sequence to control the error \cite{koshal2013,koshal2010}. 
\begin{assumption}[Vanishing step size]\label{ass_step}
The step size sequence $(\gamma_k)_{k\in\NN}$ is such that
$$\sum_{k=0}^\infty\gamma_k=\infty, \;\sum_{k=0}^\infty\gamma_k^2<\infty \text{ and }\sum_{k=0}^\infty\gamma_k^2\,\EEk{\normsq{\epsilon^k}}<\infty.\vspace{-.5cm}$$
\fineass
\end{assumption}

\subsection{Restricted strongly monotone pseudogradient}
We assume that the pseudogradient mapping in \eqref{eq_grad} satisfies the restricted strongly monotone property.

\begin{assumption}[Restricted strong monotonicity]\label{ass_res_strong}
$\FF$ is restricted $\mu$-strongly monotone at $\bs x^*\in\op{SOL}(\bs\Omega,\FF,G)$, with $\mu>0$. 
\fineass\end{assumption}

We also suppose $\FF$ to be restricted Lipschitz continuous but we remark that it is not necessary to know the actual value of the Lipschitz constant since it does not affect the parameters involved in the algorithm. This is very practical, since, in general, the Lipschitz constant is not easy to compute. 
\begin{assumption}[Restricted Lipschitz continuity]\label{ass_res_lip}
$\FF$ is restricted $\ell$-Lipschitz continuous at $\bs x^*\in\op{SOL}(\bs\Omega,\FF,G)$.
\fineass\end{assumption}

We can now state our first result.
\begin{theorem}\label{theo_SA_sne_strong}
Let Assumptions \ref{ass_error}-\ref{ass_res_lip} hold. Then, the sequence $(x^k)_{k\in\NN}$ generated by Algorithm \ref{algo}, with approximation $\hat F=F_{\textup{SA}}$ as in \eqref{eq_F_SA}, converges to a v-SNE of the game in \eqref{eq_game}.
\end{theorem}
\begin{proof}
See Appendix \ref{sec_proofs_SNE}.
\end{proof}

\begin{remark}\label{rem_strong}
A similar result can be found in \cite{jiang2008} for strongly monotone, Lipschitz continuous pseudogradient mappings on the whole feasible set, not just with respect to the solution. Moreover, Theorem \ref{theo_SA_sne_strong} extends that result to the case of nonsmooth cost functions and, compared to \cite{jiang2008}, it does not require a bound on the step size.
\end{remark}

\subsection{Restricted strictly monotone pseudogradient}

The strong monotonicity in Assumption \ref{ass_res_strong} can be replaced with restricted strict monotonicity to prove convergence, a case that we analyze in this subsection.
\begin{assumption}[Restricted strict monotonicity]\label{ass_res_strict}
$\FF$ is restricted strictly monotone at $\bs x^*\in\op{SOL}(\bs\Omega,\FF,G)$.
\fineass\end{assumption}
Also in this case, we assume that $\FF$ is restricted Lipschitz continuous (Assumption \ref{ass_res_lip}) but neither in this case it is necessary to know the value of the Lipschitz constant. 
\begin{theorem}\label{theo_SA_sne_strict}
Let Assumptions \ref{ass_error}, \ref{ass_step}, \ref{ass_res_lip} and \ref{ass_res_strict} hold. Then, the sequence $(x^k)_{k\in\NN}$ generated by Algorithm \ref{algo} with approximation $\hat F=F_{\textup{SA}}$ as in \eqref{eq_F_SA} converges a.s. to a $\mathrm{v}$-$\mathrm{SNE}$ of the game in \eqref{eq_game}.
\end{theorem}
\begin{proof}
See Appendix \ref{app_strict}.
\end{proof}
\begin{remark}
Strong monotonicity (Assumption \ref{ass_res_strong}) implies strict monotonicity (Assumption \ref{ass_res_strict}). Hence, Theorem \ref{theo_SA_sne_strict} improves Theorem \ref{theo_SA_sne_strong}.
\end{remark}
\begin{remark}
For strictly monotone and Lipschitz continuous pseudogradient mappings on the whole feasible set, a similar result can be found in \cite{koshal2010}. Moreover, Theorem \ref{theo_SA_sne_strict} extends that result to the case of nonsmooth cost functions.
\end{remark}

\section{Convergence with many samples}
Convergence of Algorithm \ref{algo} can be proven also using a different approximation scheme.

Instead of taking only one sample per iteration, let us consider the possibility of having more, even changing at each iteration. Hence, let us consider to have $S^k$ samples of the random variable $\xi$, called the batch size, and to be able to compute an approximation of $\FF(\bs x)$ of the form
\begin{equation}\label{eq_F_VR}
\begin{aligned}
\hat F(\bs x,\bs\xi)&=F_\textup{VR}(\bs x,\bs\xi)\\
&=\op{col}\left(\frac{1}{S^k} \sum_{t=1}^{S^k} \nabla_{x_i}f_i(\bs x,\xi_i^{(t)})\right)_{i\in\mc I}\\
&=\frac{1}{S^k} \sum_{t=1}^{S^k} F_\textup{SA}(\bs x,\bar\xi^{(t)})
\end{aligned}
\end{equation}
where $\bs \xi=\op{col}(\bar \xi^{(1)},\dots,\bar\xi^{(S^k)})$, for all $t=1,\dots,S^k$ $\bar\xi^{(t)}=\op{col}(\xi_1^{(t)},\dots,\xi_N^{(t)})$ and $\bs \xi$ is an i.i.d. sequence of random variables.
Approximations of the form (\ref{eq_F_VR}), where a certain number of samples is available, are common, for instance, in generative adversarial networks \cite{franci2020gan, gidel2018}. 
The subscript VR in equation \eqref{eq_F_VR} stands for \textit{variance reduction}, which is a consequence of taking an increasing batch size.
\begin{assumption}[Batch size]\label{ass_batch}
The batch size sequence $(S^k)_{k\geq 1}$ is such that, for some $c,k_0, a>0$,
$S^k\geq c(k+k_0)^{a+1}.$
\end{assumption}
This assumption implies that $1/S^k$ is summable, which is a standard assumption in this case. It is often used in combination with the assumption of uniform bounded variance \cite{bot2020,iusem2017, franci2020fbtac}. 
\begin{assumption}[Uniform bounded variance]\label{ass_variance}
For all $\bs x\in\bs \Omega$ and for some $\sigma>0$,
$\EE[\norm{F_\textup{SA}(\bs x, \cdot)-\FF(\bs x)}^{2}]\leq \sigma^2.$
\fineass\end{assumption}

\begin{remark}
It follows from Assumptions \ref{ass_batch} and \ref{ass_variance} \cite[Lemma 6]{franci2020fbtac} that for all $k\in\NN$ and some $c>0$
$$\EE[\|\epsilon^k\|^2|\mc F^k]\leq\frac{c\sigma^2}{S^k} \text{ a.s.}.$$
\end{remark}

\begin{assumption}[Step size bound]\label{ass_step_strong}
The step size sequence $(\gamma_k)_{k\in\NN}$ is such that, for all $k$, $\gamma_k\leq\frac{2\mu}{\ell^2}$ where $\mu$ is the strong monotonicity constant of $\FF$ as in Assumption \ref{ass_res_strong} and $\ell$ is its Lipschitz constant as in Assumption \ref{ass_res_lip}.
\end{assumption}
\begin{theorem}\label{theo_strong_vr}
Let Assumptions \ref{ass_error}, \ref{ass_res_strong}, \ref{ass_res_lip} and \ref{ass_batch}-\ref{ass_step_strong} hold. Then,  the sequence $(x^k)_{k\in\NN}$ generated by Algorithm \ref{algo} with approximation $\hat F=F_{\textup{VR}}$ as in \eqref{eq_F_VR} converges a.s. to a v-SNE of the game in \eqref{eq_game}.
\end{theorem}
\begin{proof}
See Appendix \ref{app_vr}.
\end{proof}
Let us note that an operator which is $\mu$-strongly monotone and $\ell$-Lipschitz continuous is also $\frac{\mu}{\ell^2}$-cocoercive, hence, we make the following assumption.
\begin{assumption}[Restricted cocoercivity]\label{ass_res_coco}
$\FF$ is restricted $\beta$-cocoercive at $\bs x^*\in\op{SOL}(\bs\Omega,\FF,G)$, with $\beta>0$. 
\end{assumption}
Let us also generalize the bound on the step size (Assumption \ref{ass_step_strong}) to this case.
\begin{assumption}[Step size bound]\label{ass_bound_step}
The step size sequence $(\gamma_k)_{k\in\NN}$ is such that, for all $k$, $\gamma_k\leq2\beta$ where $\beta$ is the cocoercivity constant of $\FF$ as in Assumption \ref{ass_res_coco}.
\fineass\end{assumption}
We note that with weaker assumptions, the FB algorithm may end in cycling behaviors and never reach a solution \cite{mescheder2018,gidel2018,grammatico2018}. Remarkably, it is sufficient for convergence that cocoercivity holds only with respect to the solution.
In this case, Lipschitz continuity is not necessary, however, a $\beta$-cocoercive mapping is also $\frac{1}{\beta}$-Lipschitz continuous.
\begin{corollary}\label{theo_VR_sne_coco}
Let Assumptions \ref{ass_error},  \ref{ass_batch}, \ref{ass_variance}, \ref{ass_res_coco} and \ref{ass_bound_step} hold. Then,  the sequence $(x^k)_{k\in\NN}$ generated by Algorithm \ref{algo} with approximation $\hat F=F_{\textup{VR}}$ as in \eqref{eq_F_VR} converges to a v-SNE of the game in \eqref{eq_game}.
\end{corollary}
\begin{proof}
It follows from \cite[Theorem 1]{franci2020fbtac}.
\end{proof}

Theorem \ref{theo_SA_sne_strict} can be extended to this approximation scheme.
\begin{corollary}\label{cor_strict}
Let Assumptions \ref{ass_error}, \ref{ass_step}, \ref{ass_res_lip} and \ref{ass_res_strict} hold. Then, the sequence $(\bs x^K)_{k\in\NN}$ generated by Algorithm \ref{algo} with approximation $\hat F=F_{\textup{VR}}$ as in \eqref{eq_F_VR} converges to a v-SNE of the game in \eqref{eq_game}.\fineass
\end{corollary}
\begin{remark}
Since we take a vanishing step size (Assumption \ref{ass_step}), in the case of Corollary \ref{cor_strict}, the batch size sequence should not be increasing, that is, we can take a constant number of realizations at every iteration.
\fineass\end{remark}
\begin{remark}
While for strong monotonicity both the approximation schemes can be used, to the best of our knowledge, convergence holds under cocoercivity only with variance reduction, and under strict monotonicity only with a finite number of samples. This is because the techniques used in \cite{franci2020fbtac} with cocoercivity cannot be applied with a vanishing step size while the hypothesis of Lemma \ref{lemma_RS}, used in \cite{koshal2010}, do not hold with a fixed step size sequence.
\fineass\end{remark}

\section{Numerical simulations}\label{sec_sim}
\subsection{Comparative example} 
We consider a two-player game with pseudogradient mapping $\FF(x)=\EE[F(x)]$, where 
$$F(x)=\left[\begin{matrix}
1 & \xi_1\\
\xi_2 & 1000
\end{matrix}\right]\left[\begin{array}{c}
x_1\\
x_2
\end{array}\right].$$
The random variables $\xi_1$ and $\xi_2$ are chosen to have mean $1$ and $1000$ respectively, hence, $\FF(x)$ is cocoercive. The decision variables are bounded, i.e., $-\theta\leq x_i\leq\theta$ for $i=1,2$ where $\theta=1000$.
We compare our SFB with the algorithms mentioned in Table \ref{table_algo}, whose iterations in compact form are reported next. 
The extragradient (SEG) \cite{kannan2019} involves two projection steps:
$$\begin{aligned} 
y^k &=\op{proj}_{\Omega}[x^k-\gamma_kF_\textup{SA}(x^{k},\xi^k)] \\ 
x^{k+1} &=\op{proj}_{\Omega}[x^k-\gamma_kF_\textup{SA}(y^k,\eta^k)].
\end{aligned}$$
The Tikhonov regularization (TIK) reads as
$$
x^{k+1}=\op{proj}_{\Omega}[x^{k}-\gamma_k(F_\textup{SA}(x^{k},\xi^k)+\epsilon^k x^{k})]
$$
where $\epsilon^k$ is the regularization parameter, chosen according to \cite[Lemma 4]{koshal2013}. The regularized smoothed stochastic approximation scheme (RSSA) has a similar iteration, but the pseudogradient depends on $z^k\in\RR^n$, a uniform random variable over the ball centered at the origin with radius $\delta^k\to 0$ \cite[Lemma 5]{yousefian2017}:
$$
x^{k+1}=\op{proj}_{\Omega}[x^k-\gamma_k(F_\textup{SA}(x^k+z^k, \xi^k)+\eta^k x^k)].
$$
Lastly, the stochastic projected reflected gradient (SPRG) \cite{cui2016} exploits the second-last iterate in the update:
$$
x^{k+1}=\op{proj}_{\Omega}[x_{k}-\gamma_{k} F_\textup{SA}(2 x_{k}-x_{k-1},\xi^k)].
$$
The step size is chosen according to $\gamma_k=(1000+k)^{-1}$.
We evaluate the performance by measuring the residual of the iterates, i.e., $\op{res}(x^k)=\|x^k-\op{prox}_g(x^k-F(x^k))\|$ which is zero if and only if $x^k$ is a solution. For the sake of a fair comparison, we run the algorithms 100 times and take the average performance. To smooth the oscillatory behavior caused by the random variable, we average each iteration over a time window of 50 iterations. As one can see form Figure \ref{fig_iter} and \ref{fig_sec}, the SFB shows faster convergence both in terms of number of iterations and computational cost.
Our numerical experience also suggests that sometimes the SPRG is faster with smaller feasible sets (i.e., smaller $\theta$).
\begin{figure}[t]
\begin{center}
\includegraphics[width=\columnwidth]{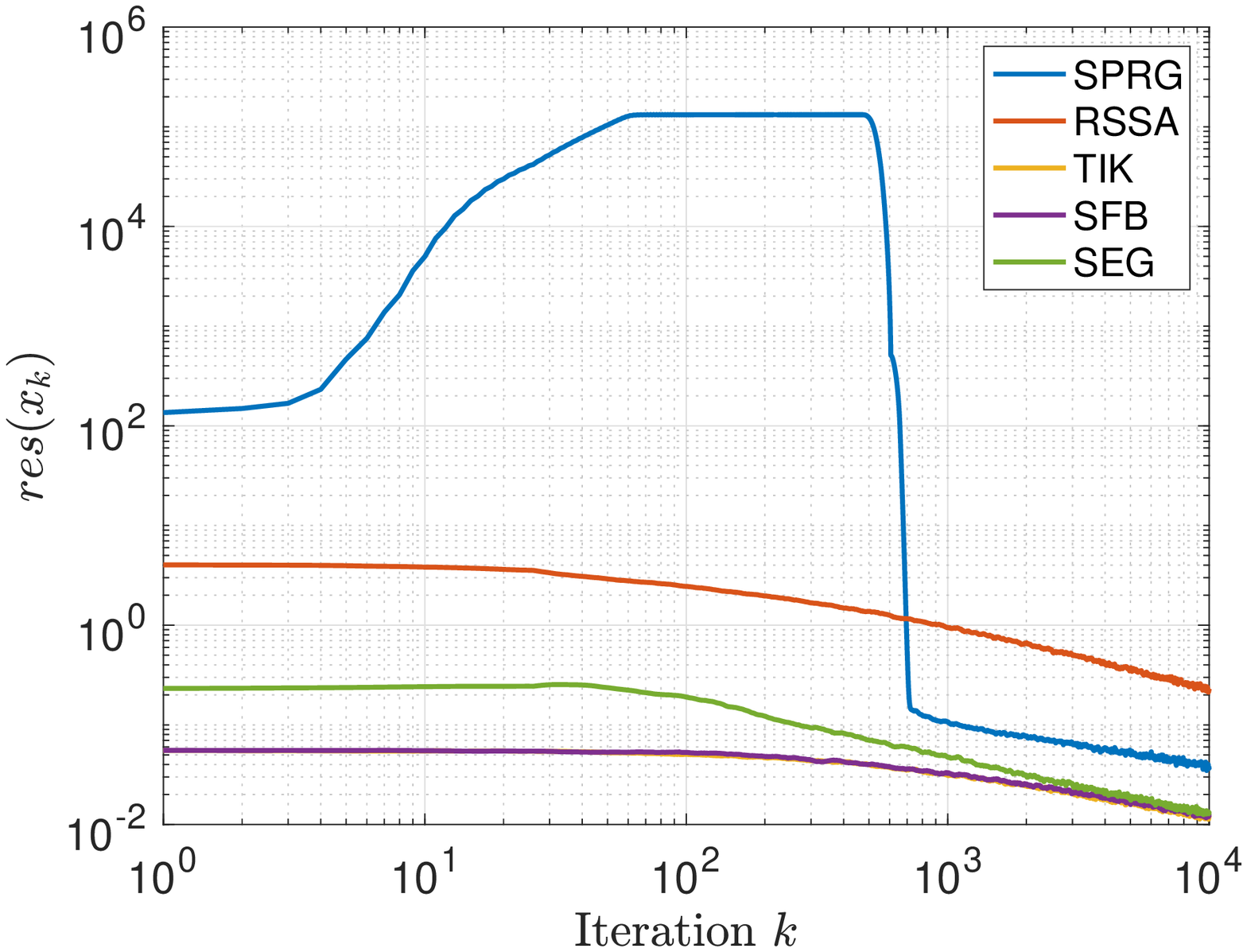}
\end{center}
\caption{Distance from the solution.}\label{fig_iter}
\end{figure}

\begin{figure}[t]
\begin{center}
\includegraphics[width=\columnwidth]{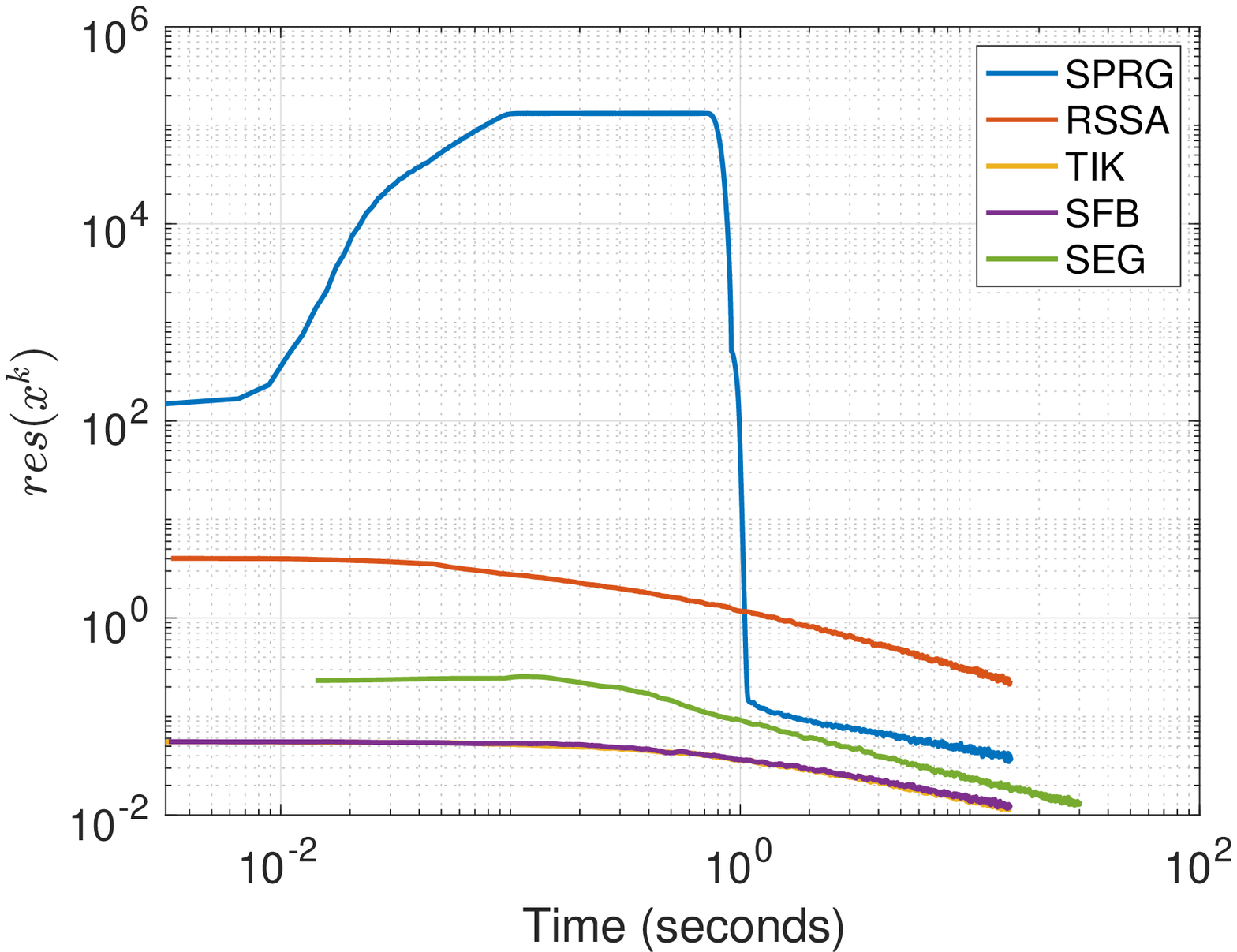}
\end{center}
\caption{Computational time.}\label{fig_sec}
\end{figure} 
\subsection{Influence of the variance}
Let us now suppose to have 3 players with pseudogradient mapping $\FF(x)=\EE[F(x)]$ where $F(x)=Fx$ and $F\in\RR^3$ is a randomly generated positive definite matrix. The uncertainty affects the elements on the antidiagonal which are iteratively drawn from a normal distribution with fixed mean, equal to the corresponding entry in $\FF(x)$ and different variances. The plot in Figure \ref{fig_10} shows, depending on the variance, the average number of iteration (over 10 simulations) at which the residual reach a precision of order $10^{-1}$ (blue), $10^{-2}$ (red) and $10^{-3}$ (green).

\begin{figure}[t]
\begin{center}
\includegraphics[width=\columnwidth]{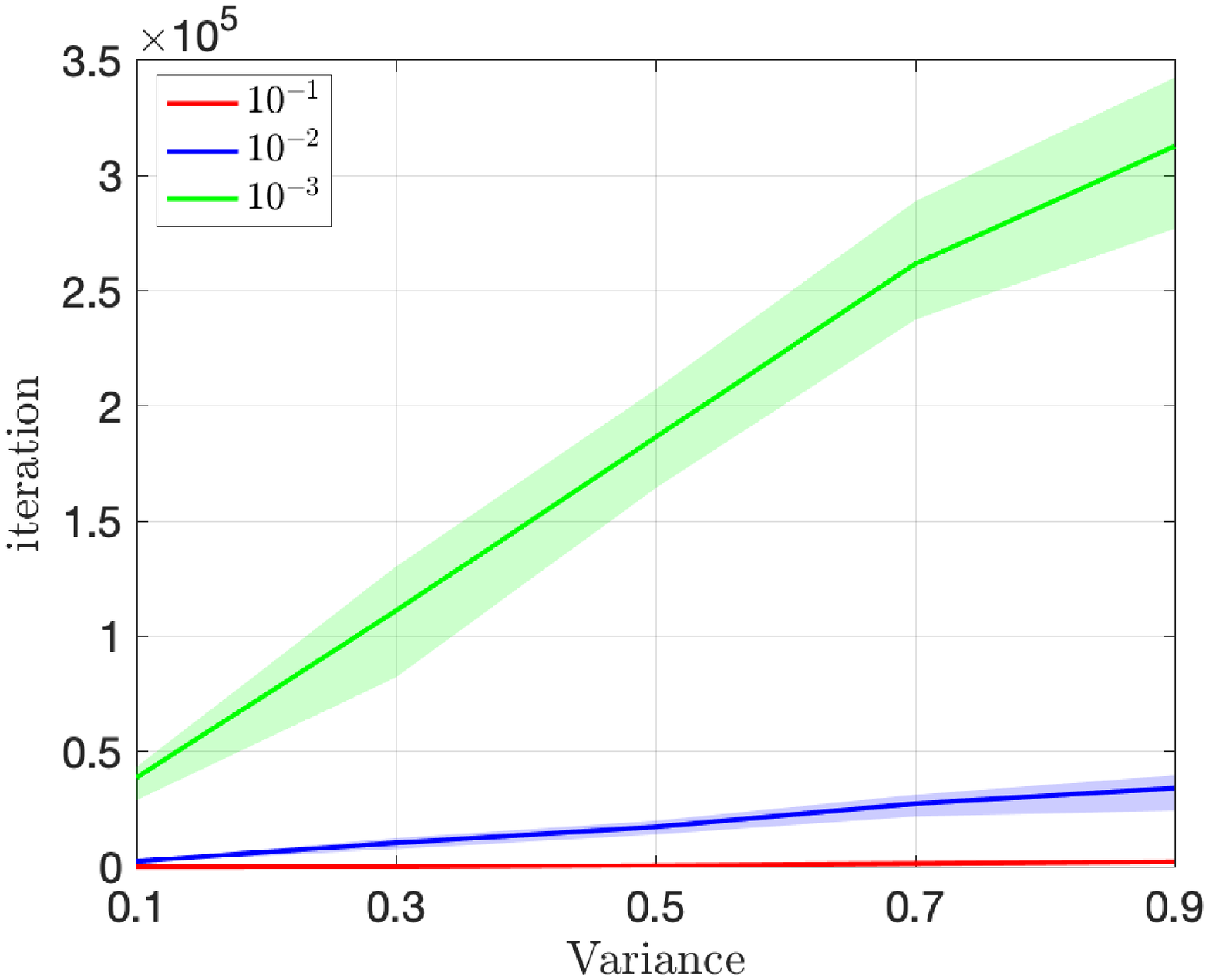}
\end{center}
\caption{Distance from the solution.}\label{fig_10}
\end{figure}
%

\section{Conclusion}\label{sec_conclu}

The stochastic forward-backward algorithm can be applied to stochastic Nash equilibrium problems with restricted monotonicity properties and nonsmooth cost functions. Specifically, convergence to a stochastic Nash equilibrium holds when the pseudogradient mapping is restricted strongly monotone or restricted strictly monotone with respect to the solution. To estimate the expected value mapping one can use both the stochastic approximation with only one sample or with variance reduction.

\appendices
\section{Auxiliary results}
Given the probability space $(\Xi, \mc F, \PP)$, let us recall some results on sequences of random variables.

First, let us define the filtration $\mc F=\{\mc F^k\}$, that is, a family of $\sigma$-algebras such that $\mathcal{F}_{0} = \sigma\left(X_{0}\right)$, for all $k \geq 1$, 
$\mathcal{F}^k = \sigma\left(X_{0}, \xi_{1}, \xi_{2}, \ldots, \xi^k\right)$
and $\mc F^k\subseteq\mc F_{k+1}$ for all $k\geq0$.

The Robbins-Siegmund Lemma is widely used in literature to prove a.s. convergence of sequences of random variables. 
\begin{lemma}[Robbins-Siegmund Lemma, \cite{RS1971}]\label{lemma_RS}
Let $\mc F=(\mc F^k)_{k\in\NN}$ be a filtration. Let $(\alpha_k)_{k\in\NN}$, $(\theta^k)_{k\in\NN}$, $(\eta^k)_{k\in\NN}$ and $(\chi^k)_{k\in\NN}$ be non negative sequences such that $\sum^k\eta^k<\infty$, $\sum^k\chi^k<\infty$ and let
$$\forall k\in\NN, \quad \EE[\alpha_{k+1}|\mc F^k]+\theta^k\leq (1+\chi^k)\alpha_k+\eta^k \quad a.s.$$
Then $\sum^k \theta^k<\infty$ and $(\alpha_k)_{k\in\NN}$ converges a.s. to a non negative random variable.
\end{lemma}

\section{Proof of Theorem \ref{theo_SA_sne_strong}}\label{sec_proofs_SNE}
\begin{proof}
Similarly to \cite[Theorem 3.1]{jiang2008}, using the nonexpansiveness of the proximal operator \cite[Proposition 12.28]{bau2011}, and Assumptions \ref{ass_res_strong} and \ref{ass_res_lip}, we have that
\begin{equation}\label{eq_proof_strong}
\begin{aligned}
&\mathbb{E}[\|x^{k+1}-x^{*}\|^{2} \mid \mathcal{F}^k] \leq\|x^{k}-x^{*}\|^{2}+\gamma_k^{2} \ell^{2}\|x^{k}-x^{*}\|^{2} \\
&-2\gamma_k\mu\|x^{k}-x^{*}\|^{2}+\gamma_k^{2} \mathbb{E}[\|w^{k}\|^{2} \mid \mathcal{F}^k]
\end{aligned}
\end{equation}
Grouping some terms we obtain
$$
\begin{aligned}
&\mathbb{E}[\|x^{k+1}-x^{*}\|^{2} \mid \mathcal{F}^k] \leq(1+\gamma_k^{2} \ell^{2})\|x^{k}-x^{*}\|^{2} \\
&-2\gamma_k\mu\|x^{k}-x^{*}\|^{2}+\gamma_k^{2} \mathbb{E}[\|w^{k}\|^{2} \mid \mathcal{F}^k]
\end{aligned}
$$
and we can apply Lemma \ref{lemma_RS} to conclude that $x^k$ is bounded and that it has at least one cluster point. Moreover, $\|x^k-x^*\|\to0$ as $k\to\infty$, hence, $x^k$ converges a.s. to $x^*$.
\end{proof}

\section{Proof of Theorem \ref{theo_SA_sne_strict}}\label{app_strict}
\begin{proof}
Similarly to \cite[Proposition 1]{koshal2010}, using the nonexpansiveness of the proximal operator \cite[Proposition 12.28]{bau2011} we have that
$$
\begin{aligned}
&\mathbb{E}[\|x^{k+1}-x^{*}\|^{2} \mid \mathcal{F}^k] \leq(1+\gamma_k^{2} L^{2})\|x^{k}-x^{*}\|^{2} \\
&+\gamma_k^{2} \mathbb{E}[\|w^{k}\|^{2} \mid \mathcal{F}^k]-2 \gamma_k\langle x^{k}-x^{*},F(x^{k})-F(x^{*})\rangle
\end{aligned}
$$
Then, applying Lemma \ref{lemma_RS} we have that $x^k$ is bounded and that it has at least one cluster point $\bar x$. Moreover, we have that $\langle x^{k}-x^{*},F(x^{k})-F(x^{*})\rangle\to 0$ and that $\langle \bar x-x^{*},F(\bar x)-F(x^{*})\rangle=0$. Since Assumption \ref{ass_res_strict} holds, $\bar x=x^*$.
\end{proof}

\section{Proof of Theorem \ref{theo_strong_vr}}\label{app_vr}
\begin{proof}
Starting from \eqref{eq_proof_strong} and grouping, we have
$$
\begin{aligned}
&\mathbb{E}[\|x^{k+1}-x^{*}\|^{2} \mid \mathcal{F}^k] \leq\|x^{k}-x^{*}\|^{2} \\
&(\gamma_k^{2} \ell^{2}-2\gamma_k\mu)\|x^{k}-x^{*}\|^{2}+\gamma_k^{2} \mathbb{E}[\|w^{k}\|^{2} \mid \mathcal{F}^k]
\end{aligned}
$$
and we can apply Lemma \ref{lemma_RS} to conclude that $x^k$ is bounded and that it has at least one cluster point. Moreover, $\|x^k-x^*\|\to0$ as $k\to\infty$, hence, $x^k$ converges a.s. to $x^*$.
\end{proof}

\bibliographystyle{IEEEtran}
\bibliography{Biblio}

\end{document}